\documentclass[english,11pt,letterpaper]{article}
\usepackage{amssymb, amsfonts, amsmath}
\usepackage{amsthm}
\usepackage{thmtools}
\usepackage{babel}
\usepackage{pdflscape}
\usepackage{graphicx}
\usepackage{hyperref}
\usepackage[ruled,vlined,linesnumbered]{algorithm2e}

\usepackage{tikz}
\usetikzlibrary{calc}

\newtheorem{theorem}{Theorem}[section]
\newtheorem{lemma}[theorem]{Lemma}
\newtheorem{corollary}[theorem]{Corollary}

\title{Analysis of Algorithms for Moser's Problems on Sums of Consecutive Primes}
\author{Jonathan P. Sorenson \\
	Department of Computer Science and Software Engineering \\
	Butler University, Indianapolis IN, USA \\
	\href{mailto:jsorenso@butler.edu}{\texttt{jsorenso@butler.edu}}
    \and Eleanor Waiss \\
	Department of Mathematical Sciences \\
	Butler University, Indianapolis IN, USA \\
	\href{mailto:ewaiss@butler.edu}{\texttt{ewaiss@butler.edu}}
}
\date{\today}

\makeatletter 
\let\c@figure\c@table
\let\ftype@figure\ftype@table
\makeatother

\begin{document}
	\maketitle

\begin{abstract}
 In his 1963 paper on the sum of consecutive primes, 
 Moser posed four open questions related to $f(n)$, 
 the number of ways an integer $n$ can be written as a sum of consecutive primes.
 (See also problem C2 from Richard K.~Guy's \textit{Unsolved Problems in Number Theory}.)
 In this paper, we present and analyze two algorithms that,
 when given a bound $x$, construct a histogram of values of $f(n)$ for all $n\le x$.
 These two algorithms were  described, but not analyzed,
 by Jean Charles Meyrignac (2000) and Michael S. Branicky (2022).
 We show the first algorithm takes $O(x\log x)$ time using
 $x^{2/3}$ space,
 and the second has two versions,
 one of which takes $O(x\log x)$ time but only
 $x^{3/5}$ space,
 and the other which takes $O(x(\log x)^2)$ time but only
 $O( \sqrt{x\log x})$ space.
 However, Meyrinac's algorithm is easier to parallelize.
 We then present data generated by these algorithms
 that address all four open questions.
\end{abstract}

    \section{Introduction\label{sec:intro}}

We say a positive integer $n$ is \textit{gleeful} if it can be
written as the sum of consecutive primes.
So, for example, $n=10$ is gleeful since $10=2+3+5$.
Some integers have multiple gleeful \textit{representations},
or distinct ways of expressing them as such a sum;
for instance, $17$ has two representations: $2+3+5+7$ and $17$
(where we consider the empty sum of a single prime to be gleeful),
and for a somewhat larger example, we have
  \begin{equation}\label{EQN:MULTREPS}
		2357 = 773 + 787 + 797 = 461 + 463 + 467 + 487.
    \end{equation}
Let $f(n)$ count the number of gleeful representations for
the integer $n$.  
Thus, $f(6)=0$, $f(10)=1$, and $f(17)=f(2357)=2$.

In his 1963 paper, Moser \cite{Moser63} proved that the average value of $f(n)$ is asymptotically equal to $\log 2$; that is,
$$\frac{1}{x} \sum_{n=1}^x f(n) = \log 2 + o(1) $$
for large $x$.
He then posed the following four questions:
\begin{enumerate}
\item Is $f(n)=1$ infinitely often?
\item Is $f(n)=k$ solvable for every integer $k\ge 0$?
\item Does the set of numbers $n$ such that $f(n)=k$ have positive
  density for every integer $k\ge 0$?
\item Is $\limsup_{n\rightarrow\infty} f(n) = \infty$?
\end{enumerate}
These questions are then reprised in problem C2 from \cite{UPINT}, and several Online Encyclopedia of Integer Sequences (OEIS) entries address these questions:
  \cite{OEIS}:
\begin{quote} 
  A054859, A054845, A067381, 
  A054996, A054997, A054998, A054999, A055000, A055001 .
\end{quote}
To address Moser's four questions,
  in this paper, we describe and analyze variations 
  of two known algorithms that generate
  data on this function $f$.
As we will demonstrate, 
  that data suggests the answers are all "yes".
In particular, given a bound $x$, our algorithms
  compute $f(n)$ for all $n\le x$ and then
  construct a histogram $h(k,x)$ 
  (oftentimes shorthanded to $h(k)$ where $x$ is implicit)
  where
  $$ h(k) = \#\{ n\le x : f(n)=k \}.$$

A first attempt at an algorithm to generate the values $h(k)$ might be like Algorithm \ref{alg:the first}.
\begin{algorithm}
    \caption{A first approach}
    \label{alg:the first}
    \KwIn{$x$}
    \KwOut{$h[\,]$}
    Allocate an array for the values of $f$
      and initialize all values to zero.\\
    Construct all sums of consecutive primes $\le x$;
      we can easily adapt the algorithm from \cite{OSS2024} to find
      these representations in time linear in $x$.\\
    For each representation found in the previous step,
       if $n$ is the sum of the primes in that representation,
       increment $f(n)$.\\
    A scan of the array holding $f$ suffices to
      create the histogram $h$.\\
\end{algorithm}
This algorithm takes time linear in $x$,
which is optimal for this problem.
But it requires space linear in $x$ to hold the array $f$
  (plus the prefix sums array - see \cite{OSS2024}).
This is problematic on currently available hardware
  if we want to compute the histogram $h$ for, say, $x=10^{14}$.
Thus, the algorithm variants we present here focus on reducing
  memory use while keeping the running time as close to $O(x)$
  as possible.

Throughout this paper, we count arithmetic operations on integers with
$O(\log x)$ bits, including branching and array indexing,
as taking constant time, and we count space
used in bits.
In practice, all of our variables are 64-bit integers or smaller. 

We now state our main results and outline the rest of this paper.

In \S\ref{sec:puzzlealg}, we look at an algorithm due to
  Jean Charles Meyrignac (2000) that appeared in
  a puzzle solution to Problem \#46, 
  posted on the \texttt{primepuzzle.net} website
  curated by Carlos Rivera \cite{puzzleweb}.
Henceforth we call this the \textit{puzzle algorithm}.
We analyze a variant of the puzzle algorithm,
  mostly by filling in the details of Meyrinac's description.
The algorithm parallelizes nicely, has
 an interesting time/space tradeoff, and
generates values of $f(n)$ for all $n$ from an interval.
  \nocite{Sorenson2015} 
  \nocite{Miller76,Rabin80}

In \S\ref{sec:pqalg}, we look at an algorithm due to 
Michael S. Branicky (2022),
which was written in Python and appears on the OEIS \cite[A054845]{OEIS}.
Again, we analyze two variants of this priority-queue-based algorithm, depending on how primes are generated.
These algorithms have the property that $f(n)$ values are
generated one at a time in increasing order based on $n$.
Although these algorithms present a better time/space tradeoff,
they resist parallelization.
Nevertheless, we were able to run one of our priority-queue variants
up to $x=10^{12}$ using a single core,
and up to $x=10^{13}$ using 156 cores.

In \S\ref{sec:comp}, we present histograms for $x$ up to
$10^{14}$.
We also discuss our implementation details.
The curious reader may find our code and data here:
\url{https://eleanorwaiss.github.io/gleeful}. 

Finally, in \S\ref{sec:moser} we give conjectured answers to
Moser's questions based on our data.

 \nocite{OSS2024,TWS2022,Sorenson2015,Galway2000,AB2004}
  

    \section{The Puzzle Algorithm\label{sec:puzzlealg}}
\newcommand{\mcutoff}{{m_\mathit{cutoff}}}
\newcommand{\mcutoffb}{{m_{\mathit{cutoff}}^2}}
\newcommand{\mcutoffc}{{m_{\mathit{cutoff}}^3}}
\newcommand{\slide}{{\mathbf{slide}}}

The \textit{puzzle algorithm}, as we call it,
is based on the following description,
written by Jean Charles Meyrignac \cite{puzzleweb}:

 \begin{quote}\small

\begin{verbatim}
1) First, pre-compute a table of primes 
2) Lower bound = 0
3) Higher Bound=Lower Bound + size of sliding window
4) Compute all consecutive sums of n elements for every element
of the pre-computed table and counts every occurrence.
5) Display the best occurrences
6) Change the lower bound and go to 3)
\end{verbatim}

\end{quote}
Our version of this algorithm varies in several ways, but
the main idea is the same.
Values of $f(n)$ are computed for every integer $n$ in
a segment or window, and then these are scanned to add to
a running histogram.
We spend some time discussing the details of how
step (4) above can be done efficiently.
But before we dive into the details, we need to
establish some notation and prove some useful lemmas.

\subsection{Notation and Lemmas}

Let us introduce some notation.

Let $x$ be the limit for our computation.
The interval $[1,x]$ is divided into equal-length segments
(or windows)
of size $\Delta$ each (the \textit{sliding window}).
The last segment may be shorter if $\Delta$ does not divide $x$.
There are $\lceil x/\Delta\rceil$ segments total.
For a particular segment, let $x_1$ be the left endpoint and
$x_2$ be the right endpoint, but just outside the window.
We have $\Delta=x_2-x_1$, and we assume that $\Delta$ will
be a fractional power of $x$, such as $\sqrt{x}$ or $x^{2/3}$.

The basic idea is that the segment represents the values
of $f(n)$ for $x_1\le n<x_2$, which takes $\Delta$ small
integers to represent (an 8-bit integer suffices in practice).

We use the term \textit{chain} to indicate a string of
consecutive primes, and we use $m$ to denote the
number of primes, or length, of the chain.
Then $n$ is the sum of those primes.
If we write $c$ as a variable representing a chain, 
then $m(c)$ is its length and $n(c)$ is its sum.
In practice, we store the start and stop prime indices,
the length, and the sum in a chain variable.

For example, if $c:=[2,3,5,7]$, then $m(c)=4$ and
$n(c)=17$.
We will often wish to \textit{slide} a chain, which means
dropping its smallest prime and adding the next larger prime.
A slide operation preserves length but increases the sum.
Thus, $\slide(c)=[3,5,7,11]$.
If a list of primes has been precomputed,
we can implement a $\slide$ operation so that it takes
constant time.

Let $M:=M(x)$ be the length of the longest chain of
consecutive primes whose sum is $\le x$.
WLOG we may assume $2$ is in the chain so that 
$$
2+3+5+\cdots+p_M \le x
< 2+3+5+\cdots+p_M+p_{M+1}.
$$
For any chain of length $m$ with sum $n\le x$,
we have $m\le M$.

  Let $p(m)=p(m,x)$ be the largest prime that participates in any
  chain with sum $n\le x$ of length $m$.
  \begin{lemma}\label{lemma:pm}
  We have
  $$ \frac{x}{m+2} < p(m,x) \le \frac{2x}{m}(1+o(1)).$$
  \end{lemma}
  \begin{proof}
      There exists an integer $\ell>0$ such that
      $$
      p_{\ell+1}+p_{\ell+2}+\cdots +p_{\ell+m} \le x
      < p_{\ell+1}+p_{\ell+2}+\cdots +p_{\ell+m} + p_{\ell+m+1},
      $$
      so that $p(m,x)=p_{\ell+m}$ here.

      For the lower bound, we have
      \begin{eqnarray*}
      x
      &<& p_{\ell+1}+p_{\ell+2}+\cdots +p_{\ell+m} + p_{\ell+m+1} \\
      &\le& p_{\ell+1}+p_{\ell+2}+\cdots +p_{\ell+m} + 2p_{\ell+m} \\
      &\le& (m+2) p_{\ell+m} = (m+2) p(m)
      \end{eqnarray*}
      by Bertrand's postulate.

      For the upper bound, we note that
      \begin{equation}\label{EQ:BS}
      \sum_{i=1}^u p_i \sim \sum_{i=1}^{u} i \log i
       \sim \frac{u^2 \log u}{2}
      \end{equation}
      (see Bach and Shallit \cite{BS}, chapter 2).
      We then have
      \begin{eqnarray*}
          x&\ge&
          p_{\ell+1}+p_{\ell+2}+\cdots +p_{\ell+m} \\
          &\sim& \frac{1}{2} ( (\ell+m)^2\log(\ell+m)-\ell^2\log \ell) \\
          &\sim& \frac{m}{2} p_{\ell+m} + \frac{\ell}{2}(p_{\ell+m}-p_\ell) \\
          &\ge& \frac{m}{2} p(m),
      \end{eqnarray*}
      by the prime number theorem.
  \end{proof}

 We have the following estimate for $M=M(x)$:
    \begin{lemma}\label{lemma:M}
        $$M(x)\sim \frac{2\sqrt{x}}{\sqrt{\log x}}.$$
    \end{lemma}
    \begin{proof}
    Consider a chain of length $M$ that  
    consists of the first $M$ primes.
        From (\ref{EQ:BS}), we have
        $$
        \frac{M^2\log M}{2} \le x(1+o(1)) < \frac{(M+1)^2\log (M+1)}{2}.
        $$
        Taking logarithms, we deduce that $\log x \sim 2\log M$.
        The result then follows.
    \end{proof}

    And we have the following bound on sums of slided chains.

    \begin{lemma}\label{lemma:slide}
        If $c$ is a chain of length $m$ and sum $n$,
        then
        $n(\slide(c))-n(c)$
        has average value $\Theta( m \log (n/m) )$.
    \end{lemma}
    \begin{proof}
        A single $\slide$ operation changes the sum of the chain
    by $p-q$, where $p$
    is the newest prime added to the chain and $q$
    was the smallest prime just removed from the chain
    by sliding.
    By Lemma \ref{lemma:pm}, we know $p=\Theta(n/m)$,
    and so by the prime number theorem, 
    we know consecutive primes in the chain are, on average,
    $\log(n/m)(1+o(1))$ apart.
    Thus $p \approx q+m\log(n/m)$, or
$$
n(\slide(c))-n(c) = p-q= \Theta( m\log(n/m))
$$
on average.
    \end{proof}
    \begin{corollary}\label{cor:work}
        The number of chain sums $n$ from chains of length $m$
        that fall in a segment
        of length $\Delta$ is, on average,
        proportional to $$\Delta/(m\log (x/m)),$$
        where $x$ is the midpoint of the interval.
    \end{corollary}


    \subsection{Algorithm Description}

    We are now ready to describe our version of the puzzle algorithm.  The input is a bound $x$,
    and the output is an array representing the histogram $h$, where $h(k)=\#\{ n\le x : f(n)=k\}$.

    In step (4) of Meyrignac's description,
    the basic idea is to work through the chains from maximum
    length $M$ down to length 1, generating all relevant chains of
    the current length $m$ before moving on.
    If the length $m$ is "large", we build the first chain
    of length $m$ from the
    previous first chain of length $m+1$ by dropping the largest prime
    and then sliding as needed.
    (If we were to build the chain of length $m$ from scratch,
    this takes $\sum m \approx M^2$ time per segment,
    which is way too slow.)
    
    The original chain of length $M$ is constructed by
    finding primes until the sum is maximal but less than $x_2$.

    When $m$ gets small enough, the first chain of length $m$
    has no overlap with the previous one of length $m+1$,
    and may even be far enough away that making the
    first $m$-length chain from
    scratch is more efficient.
    This is how we define "small" values of $m$.
    We define $\mcutoff$ as the largest "small" value of $m$.

    To set a more precise value for $\mcutoff$,
    we set the work threshold
    to the the time needed to process the segment for that
    value of $m$ after that first chain is found. 
    By Corollary \ref{cor:work}, this is
    $O(\Delta/(\mcutoff\log x))$ operations.

    A key in the analysis is estimating this cutoff point,
    which we do now.

    \begin{lemma}\label{lemma:mcutoff} On average, we have
        $$ \mcutoff = \Theta( x_1/\Delta ).$$
    \end{lemma}
    \begin{proof}
        Let $c_a$ be the first chain of length $m$ we use (after),
        and $c_b$ be the first chain of length $m+1$ (before).
        Both chains have sums near $x_1$, as in
        \begin{eqnarray*}
            n(c_b)\ge x_1 &\mbox{ and }&n( \slide^{-1}(c_b))<x_1 \\
             n(c_a)\ge x_1 &\mbox{ and }&n( \slide^{-1}(c_a))<x_1 
        \end{eqnarray*}
        are both true.
        By Lemma \ref{lemma:slide}, we can deduce that
        \begin{equation} \label{eq:ablimit}
            |n(c_a)-n(c_b)|=O(m\log (x/m))
        \end{equation}
        on average.
        
        By Lemma \ref{lemma:pm}, the largest prime
        in $c_b$ is $p(m+1,x_1)=\Theta(x_1/m)$.
        Let $s$ denote the number of $\slide$ operations to
        obtain $c_a$ from $c_b$ after $p(m+1,x_1)$ is removed
        from $c_b$.  Then we have
        $$
        n(c_a) = n(\slide^s(c_b))-p(m+1,x_1).
        $$
        Note that removing $p(m+1,x_1)$ before or after sliding
        gives the same sum.

        Sliding a chain of length $m$ with sum near $x$
        changes its sum by $\Theta(m\log(x/m))$ by 
        Lemma \ref{lemma:slide}.
        This gives us
        $$
        n(c_a) = n(c_b)
        +\Theta(sm\log(x_1/m))-p(m+1,x_1)
        $$
        or
        $$
        p(m+1,x_1)=  n(c_b)-n(c_a)
        +\Theta(sm\log(x_1/m))
        $$
        on average.
        Using Equation (\ref{eq:ablimit}) then gives
        $$p(m+1,x_1)=\Theta( sm\log (x_1/m) )$$
        or, with Lemma \ref{lemma:pm},
        $$
        \frac{x_1}{\log x_1} \approx s\cdot \mcutoff2
        $$
        on average for when $m=\mcutoff$.
        But, we have
        $$
        s=\Theta\left(\frac{\Delta}{\mcutoff\log x}\right),
        $$
        as the number of slide operations matches the work
        for the segment for that value of $m$
        by Corollary \ref{cor:work}.
        A bit of simplification completes the proof.        
    \end{proof}

    With an estimate for $\mcutoff$ in hand, we are now
    ready to describe the Puzzle Algorithm is detail (Algorithm \ref{alg:puzzle}).


\newcommand{\esum}{\texttt{sum}}
\newcommand{\len}{\texttt{length}}
\newcommand{\pmin}{\texttt{pmin}}

\begin{algorithm}
    \label{alg:puzzle}
    \caption{The Puzzle Algorithm}
    \KwIn{$x$, $\Delta$}
    \KwOut{$h[\,]$}
    Initialize $h:=[\,]$, of length 20, to zeroes\;
    Construct primes up to $O(\Delta)$\; \label{line:largeprimes} 
    Initialize $f:=[\,]$, of length $\Delta$, to zeroes\;
    $x_1:=2, x_2:=\Delta+x_1$\; \label{line:lastsetupline} 
    \While{$x_1\le x$}{ \label{line:mainloop} 
        \lIf{$x_2>x+1$}{$x_2:=x+1$}
        Clear $f$ to zeroes\;
        Build chain $c$ such that $c.\len\leq M(x_2)$, $c.\esum\leq x_2$, $c.\pmin:=2$, and $2+\slide(c).\esum>x_2$\;
        $\mcutoff:=0$\;
        \While{$c.\len > \mcutoff$}{ \label{line:fc1}
            $s:=\Delta/(c.\len\cdot\log(x_1/c.\len))$\;
            $i:=0$\;
            \While{$c.\esum<x_1$}{ \label{line:rwis1} 
                Increment $i$\;
                $c:=\slide(c)$\;
                \lIf{$i\geq s$}{$\mcutoff:=m$}
            }
            $c_0:=c$\;
            \While{$c.\esum < x_2$}{ \label{line:rwis2}
                Increment $f[c.\esum-x_1]$\;
                $c:=\slide(c)$\;
            }
            $c:=c_0$\;
            Drop the largest prime from $c$\; \label{line:remove}
        }
        $m:=\mcutoff-1$\;
        \While{m>0}{ \label{line:fc2}
            Find all primes in $[p(x_1,m)-m\log(x_2),p(x_2,m)]$\; \label{line:smallprimes} 
            Build chain $c$ such that $c.\len=m$, $x_1\leq c.\esum \leq x_2$, and $\slide^{-1}(c).\esum < x_1$\;  \label{line:build}
            \While{$c.\esum < x_2$}{ \label{line:rwis3} 
                Increment $f[c.\esum-x_1]$\;
                $c:=\slide(c)$\;
            }
            Decrement $m$\;
        }
        \lFor{each $n$, $x_1\leq n<x_2$}{Increment $h[f[n-x_1]]\;$}
        $x_1:=x_2$\;
        $x_2:=x_2+\Delta$\;
    }
\end{algorithm}

\subsection{Time and space analysis}

In our analysis, we assume that $\Delta$ is a fractional
power of $x$ and, in fact, at least $M(x)$.
We assume that $x_2\approx x$.
The number of loop iterations when $x_2$ is substantially smaller
than $x$ is limited, and we can assume, therefore, that this case
contributes negligibly.

The space used by our algorithm will be at least $\Delta$ bits.
If possible, we want to keep space usage as low as possible
while keeping our running time as close to linear in $x$
as we can.

Since the average value of $f(n)$ is constant, most values of
$f(n)$ can be represented in a constant number of bits,
with larger values stored in a secondary data structure
to keep the overall space use linear.
In practice, we simply used an array of 8-bit integers.
We have not yet found an integer $n$ with $f(n)>14$.

Throughout, we mention the use of average-case analysis.
For example, consecutive primes near $x$ differ by $\log x$ on average, by
the prime number theorem.
For a particular step on a particular segment, this will vary, but over the course of
the entire algorithm, significant variations from the average must cancel out,
as the algorithm examines all primes and all chains up to $x$.

\subsubsection{Preprocessing}
It only takes $O(\Delta)$ time to do steps
  1 to \ref{line:lastsetupline}.
In particular, finding the primes up to $O(\Delta)$
  can be done in linear time using any number of prime
  sieves, including the sieve of Eratosthenes with a wheel
  or the Atkin-Bernstein sieve \cite{AB2004}.
In practice, we use an incremental sieve \cite{Sorenson2015}
  to find more primes as needed,
  since we don't know the implied constant in $O(\Delta)$.

\subsubsection{Main Loop}
We move on to the the main loop,
  Line \ref{line:mainloop}, which repeats $O(x/\Delta)$ times.

Inside the main loop, all the steps outside of processing
  large and small values of $m$ can be easily seen to take
  at most $O(\Delta)$ time, especially when we recall
  that $M\ll \Delta$.

\subsubsection{Large $m$}
Processing chains with large $m$ begins on Line \ref{line:fc1}. 
Everything, aside from the loops starting on Lines \ref{line:rwis1} and
\ref{line:rwis2}, takes constant time.
Summing over all large $m$ gives us $O(M)=O(\Delta)$.

By Corollary \ref{cor:work}, the loop at Line
\ref{line:rwis2} takes at most
$O(\Delta/(m\log(x/m)))$ time.
We can simplify this to $O(\Delta/(m\log x))$
because $1\le m \le M=\Theta(\sqrt{x/\log x}$ by 
Lemma \ref{lemma:M}, so that $\log x = \Theta(\log(x/m))$.

The value of $s$ in Line \ref{line:rwis1} is chosen
specifically to make sure that the cost of
the loop at Line \ref{line:rwis1} is not more than the cost
of the one at Line \ref{line:rwis2}.

For large $m$, then, we have a cost proportional to
$$
\frac{x}{\Delta} \left(
\Delta + \sum_{m=\mcutoff}^M \frac{\Delta}{m\log x}
\right).
$$
By Lemma \ref{lemma:M}, this simplifies to $O(x)$.

\subsubsection{Small $m$}
Processing chains with small $m$ begins on Line \ref{line:fc2}. 
For small $m$, the largest prime of any chain of length $m$
will range from $p(m,x_1)$ to $p(m,x_2)$.
This explains the need for Line \ref{line:smallprimes}.
By Lemma \ref{lemma:pm} this is from roughly $x_1/m$ to
$x_2/m$, or an interval of size $O(\Delta/m)$.
The smaller primes of any chain either fall in that interval,
or just below it by roughly $m\log x$ by the prime number
theorem.

Now, the Atkin-Bernstein sieve will find primes of size
$y$ on an interval of size $y^{1/3}$ or larger in time linear in 
the interval length,
using Galway's improvements \cite{Galway2000}.
Here, $y=p(m,x_2)=O(x/m)$, so we need
\begin{eqnarray*}
    \frac{\Delta}{m}+m\log x &\gg& (x/m)^{1/3}
    \quad \mbox{ or }\\
    m &\ll& \Delta^{3/2}/x^{1/2}.
\end{eqnarray*}
Using Lemma \ref{lemma:mcutoff} and a bit of work,
this gives us an upper bound of $x^{2/5}$ for $\mcutoff$,
as $m\le \mcutoff$ for "small" $m$.
As we will see later, $\mcutoff\approx x^{1/3}$,
and so the Atkin-Bernstein sieve with Galway's improvements
will work fine here.
(Note that in practice we used something akin to the
pseudosquares prime sieve \cite{Sorenson06}, 
but with Miller-Rabin prime tests
using particular bases to achieve deterministic primality testing.)

Summing this over all segments and over all $m\le \mcutoff$ gives
a time bound of
\begin{eqnarray*}
    \frac{x}{\Delta} \sum_{m=1}^\mcutoff
    \left(\frac{\Delta}{m}+m\log x\right)
    &\sim& x \log \mcutoff + \frac{x\log x}{2\Delta} \cdot
    m^2_{\mathit{cutoff}}
\end{eqnarray*}
To keep this close to linear, we want
$$
\Delta \gg m^2_{\mathit{cutoff}} .
$$

Line \ref{line:build} can be done in $O(m)$ time now
that the necessary primes are available.

Line \ref{line:rwis3} takes at most
$O( \Delta/(m\log x))$ time by Corollary \ref{cor:work},
since the primes are available so that
$\slide$ operations take constant time.

\begin{theorem}\label{thm:puzzle}
  On input $x>0$,
    the Puzzle Algorithm computes a histogram $h$ as described above
    in $O(x\log x)$ arithmetic operations and
    $O( x^{2/3})$ space.
\end{theorem}
\begin{proof}
From our discussion above,
everything outside processing small values of $m$
takes $O(x)$ time.

The equation to describe the asymptotic
running time, up to a constant multiplicative factor,
for small values of $m$, is
\begin{eqnarray*}
&&
\frac{x}{\Delta} 
\sum_{m=1}^{\mcutoff}\left(\frac{\Delta}{m}+ m \log x
+ m+\frac{\Delta}{m\cdot\log x}
\right) \\
&\sim& 
x \log \mcutoff+ \frac{\mcutoffb\cdot x\log x}{2\Delta}
+ \frac{\mcutoffb\cdot x}{2\Delta}+\frac{x\log \mcutoff}{\log x}\\
&\sim&
x \log \mcutoff+ \frac{\mcutoffb\cdot x\log x}{2\Delta}
.
\end{eqnarray*}
This argues for making $\mcutoff$ as small as possible,
but we have Lemma \ref{lemma:mcutoff}.
Multiplying the numerator and denominator of the
second term by $\mcutoff$ and writing 
$\Delta\cdot \mcutoff=\Theta( x)$ gives
$$
x \log \mcutoff + \mcutoffc \cdot \log x.
$$
We obtain that $\mcutoff=x^{1/3}$ and $\Delta=x^{2/3}$
and the proof is complete.
    
\end{proof}

On current hardware, $x^{2/3}$ space is problematic if we want to compute
  the histogram $h$ up to, say, $x=10^{15}$, as this means using $10^{10}$ space, 
  which is too much.
What happens if we allow our runtime to creep up - does this
  allow for less space use?
The answer is yes, as the runtime is inversely proportional to $\Delta$.
If we set $\Delta=x^{2/3-u}$, 
then by Lemma \ref{lemma:mcutoff} we have $\mcutoff=O(x/\Delta)=x^{1/3+u}$,
then the runtime becomes
$$
O\left(
x\log x + \frac{(x^{1/3+u})^2 x\log x}{x^{2/3-u}}
\right) = O( x^{1+3u}\log x).
$$
So for example, setting $\Delta=\sqrt{x}$, or $u=1/6$, gives a running time
of $O(x^{3/2}\log x)$.

    \section{The Priority Queue Algorithm\label{sec:pqalg}}

The main idea behind the \textit{priority queue} algorithm
is to generate all the chains with sums between $1$ and $x$,
our upper bound, and insert them into a priority queue
(a min-heap)
using their sums as their priority.
Thus, the front of the queue will have the chains of
smallest sums.
Then we can find all chains that sum to 1
from the front of the queue, 
counting them as we remove them from the priority queue,
and incrementing the histogram as appropriate. 
Rinse and repeat with $2,3,4,\ldots,x$.
This works, but there are too many chains to store all
at once -- $O(x)$ according to Moser.
However, observe that a particular value of $n$ 
has at most one representation of any given length.
This means we need only store one chain of each length
in the priority queue.
By Lemma \ref{lemma:M} we know that this is
 $O(\sqrt{x/\log x})$, much better than $O(x)$.
When a chain is removed from the priority queue, 
we simply perform a
$\slide$ and then insert it back into the queue.

This is precisely the idea in Michael S.~Branicky's
algorithm, described as Python code
from \cite[A054845]{OEIS} in 2022:
\begin{quote}\small
\begin{verbatim}
(Python) # alternate version for going to large n
import heapq
from sympy import prime
from itertools import islice
def agen(): # generator of terms
    p = v = prime(1); h = [(p, 1, 1)]; nextcount = 2; oldv = ways = 0
    while True:
        (v, s, l) = heapq.heappop(h)
        if v == oldv: ways += 1
        else:
            yield ways
            for n in range(oldv+1, v): yield 0
            ways = 1
        if v >= p:
            p += prime(nextcount)
            heapq.heappush(h, (p, 1, nextcount))
            nextcount += 1
        oldv = v
        v -= prime(s); s += 1; l += 1; v += prime(l)
        heapq.heappush(h, (v, s, l))
print(list(islice(agen(), 102))) # Michael S. Branicky, Feb 17 2022
\end{verbatim}
\end{quote}
The reader need not examine this Python program too closely,
as we will restate the algorithm, using the notation
in this paper, elaborating on details, and analyzing
the running time as we go.

\subsection{Priorty Queue Algorithm}

Let the interval $[x_1,x_2)$ be the input to the algorithm.
For ease of notation, let $\Delta:=x_2-x_1$.
As in the last section, the output is a histogram
$h$ where $h[k]=\#\{n: x_1\le n<x_2 \mbox{ and } f(n)=k\}$.

It is fine to keep in mind $[1,x+1)$ as the input for now,
but we are setting the stage for parallelization later.

\subsubsection{Setup}

\newcommand{\happy}{{\texttt{happy}}}
\newcommand{\happysum}{{\texttt{sum}}}
\newcommand{\happypmin}{{\texttt{pmin}}}
\newcommand{\happypmax}{{\texttt{pmax}}}

A chain object is stored as $3$ integers:
the smallest prime of the chain (\happypmin), 
the largest prime of the chain (\happypmax), 
and its sum (\happysum).
We store chain objects in an array, and the array position
indicates its length, so the length of the chain is
stored implicitly, and array position zero is left empty.

To create this array of chain objects,
we start by computing the longest chain $c$
that includes $2$, with maximal sum $<x_2$.
Let $M$ be the length of $c$.
Allocate the array $\happy[\,]$ of length $M+1$,
and set $\happy[M]:=c$.

We then fill the array from position $M-1$ down to $1$.
For long chains of length $m$, 
we copy the previous chain of length $m+1$, drop the 
largest prime, and then $\slide$ it as necessary
until its sum is minimal, but $\ge x_1$.
For short chains of length $m$, we estimate the largest prime of the
chain to be $x_1/m$ (see Lemma \ref{lemma:pm}),
find consecutive primes below that to get $m$ of them in total,
and then $\slide$ as necessary to get a sum that is
minimal but $\ge x_1$.

We prove below that if we split long from short chains 
at $x_2^{1/3}$, the number of $\slide$ operations it takes
to generate all the chains is $O(x_2^{2/3})$.
This also tells us that we want $x_2-x_1\gg x_2^{2/3}$
for the setup time to take negligible time compared to our main loop.
Unlike the puzzle algorithm, we do not need any space for the interval.
\begin{lemma}\label{lemma:pqsetup}
    A list of chains, one of each length from $1$ to $M(x_2)$,
    where the chain of length $m$ has minimal sum $\ge x_1$,
    can be constructed in a total number of $\slide$ operations
    that does not exceed $O(\sqrt{x_2/\log x_2}+x_1^{2/3})$.
\end{lemma}
\begin{proof}
    Let $m_c:=x_1^{1/3}$ be the cutoff length between long and short chains
    as described above.

    The first chain of length $M(x_2)$ can be made with
    $M(x_2)=O(\sqrt{x_2/\log x_2})$ $\slide$ operations or the equivalent
    by Lemma \ref{lemma:M}.

    Each subsequent long chain of length $m$ is created from the previous chain
    of length $m+1$ as described above, and this takes
    $$1+\pi(p(m,x_1))-\pi(p(m+1,x_1))$$
    $\slide$ operations, where $\pi(x)$ is the prime counting function.
    This telescopes when summed, giving 
    \begin{eqnarray*}
        \sum_{m=m_c}^{M(x_2)} (1+\pi(p(m,x_1))-\pi(p(m+1,x_1)))
    &\le& O(M)+\pi(p(m_c,x_1))-\pi(p(M,x_1)) \\
    &=& O(M+x_1/m_c)
    \end{eqnarray*}
    by Lemma \ref{lemma:pm} and the prime number theorem.

    A short chain of length $m$ needs $m$ primes near $x_1/m$,
    the equivalent of $O(m)$ $\slide$ operations.
    Summing over all short lengths $m$ gives $O(m_c^2)$ total.

    We then optimize the cost $O(M+x_1/m_c+m_c^2)$ giving $m_c=x_1^{1/3}$
    and our proof is complete.
\end{proof}

We conclude the setup by inserting the integers $1$ to $M$
into the priority queue, where the priority of the
integer $i$ is $\happy[i].\happysum$.

Later, we discuss two possible implementations of the
priority queue: as a binary heap, or using a hash table.
In either case, the initial queue can be constructed
in $O(M)$ time: We can use a \textit{build} operation
for the heap, or $O(M)$ hash table insertions of chain objects,
depending on our choice for data structure.

\subsubsection{Main Loop}

We describe the main loop of our algorithm below
(Algorithm \ref{alg:pq}).

\begin{algorithm}[ht]
    \label{alg:pq}
    \caption{Main Loop for the Priority Queue Algorithm}
    \For{ $n:=x_1$ to $x_2-1$}{
        $f:=0$\;
        \While{there are objects with priority $n$ in the priority queue:}{ \label{line:pqwhile}
            Let $c$ be the chain object at the front of the queue with $c.\happysum=n$\;
            Remove $c$ from the priority queue\;
            Increment $f$\;
            $c:=\slide(c)$\;
            \lIf{$c.\happysum<x_2$}{ Insert $c$ back into the priority queue;}
        }
        Increment $h[f]$\;
    }
\end{algorithm}

The inner While-loop, which starts at Line \ref{line:pqwhile}, iterates once per chain, which is once per
gleeful representation, which is linear, or $O(\Delta)$ on average,
over the course of the algorithm, due to Moser's theorem.
We have one delete-Min or dequeue operation, at most one insertion, and one $\slide$.

If we use a hash table to store the priority queue (where the key is the chain sum),
insertions and deletions are constant time on average.
If we use a standard binary heap, these operations take $O(\log M)$ time each.
As $M<\sqrt{x_2}$, this is $O(\log x_2)$, for a total of
$O(\Delta\log x_2)$ time for the interval.

Next, we discuss the cost of $\slide$ operations.
If a list of primes is available, we can store the index of the prime
in our chain objects, and then finding the next prime (for both
\happypmax\ and \happypmin) takes constant time.
However, we need primes as large as $x_2$, and finding all primes up to $x_2$
takes too much time, and storing them takes too much space.
But we can pick a bound $B$, and compute and store all primes up to $B$.
For primes needed beyond $B$, we can generate them either using an
incremental prime sieve (see \cite{Sorenson2015}) or using prime tests,
with some support to make this as efficient as possible (see \cite{Sorenson06}).

Let us now derive optimal values for $B$, based on whether we use
an incremental sieve or prime tests.

We begin with an incremental sieve.
Galway \cite{Galway2000} showed how to modify the Atkin-Bernstein sieve \cite{AB2004}
to find primes in an interval $[a,b]$ using at most $O(b-a+b^{1/3})$ time
and $O(b^{1/3})$ space.  This sieve can, theoretically, be constructed to
generate the primes as needed \cite{Sorenson2015}.
We divide the chains stored in our priority queue into those that use primes
up to $B$, and those that use this incremental sieve to generate needed primes.
Thus, each chain object that uses the sieve, would have attached to it its own
copy of an Atkin-Bernstein-Galway incremental sieve to generate primes as needed
for sliding.
We now prove that the optimal value for $B$ is $x_2^{3/5}$ in this case.

\begin{theorem}\label{thm:pqsieve}
After setup,
    the priority queue algorithm takes $O(\Delta\log x_2+x_2^{3/5})$ time
    to find and count all representations of integers in the interval
    $[x_1,x_2)$ where $\Delta=x_2-x_1$.
    The space used is $O(x_2^{3/5})$.
\end{theorem}
In practice, we want $\Delta\gg x_1^{2/3}$ so that the setup cost
is negligible.
Thus, for the interval $(1,x]$ we get a running time of $O(x\log x)$
using $O(x^{3/5})$ space.
We can parallelize this keeping an optimal amount of work using
up to $O(x^{1/3})$ processors, where each processor gets its own interval
of length $x^{2/3}$.
\begin{proof}
  Set $B:=x_2^{3/5}$.
    The algorithm starts by finding all primes up to $O(B)$.
    Any linear-time prime sieve is fine for this.
    The primes can be stored in $O(B)$ bits.

    By Lemma \ref{lemma:pqsetup} the time for initializing the priority queue
    is negligible.

    For $\slide$ operations on chains whose primes fall below $B$, $\slide$
    operations take constant time.  
    No extra space beyond the $O(B)$ to store the prime list is required.
    Let $m$ be the length of any such chain.  Then $p(m,x_2)\le B$.
    By Lemma \ref{lemma:pm}, we have $p(m,x_2)=\Theta(x_2/m)$, giving
    $m\ge \Theta(x_2/B) = \Theta(x_2^{2/5})$.

    Chains of length $m$ with  $m\le\Theta(x_2^{2/5})$, then, must have an
    incremental sieve attached.
    Each one uses space $O( p(m,x_2)^{1/3})$.  Again using Lemma \ref{lemma:pm},
    the total space used is proportional to
    $$
    \sum_{m=1}^{x_2^{2/5}} \left(\frac{x_2}{m}\right)^{1/3}
    = O(x_2^{1/3} \cdot (x_2^{2/5})^{2/3}) = O(x_2^{3/5}).
    $$
    For these shorter chains, $\slide$ operations take, on average, $O(\log x_2)$
    time each, as the Atkin-Bernstein-Galway sieve takes linear time to find
    primes up to a given bound, and primes are, on average, $O(\log x_2)$ apart.
\end{proof}
    If an overwhelming majority of chains are long, we might pull the average
    cost per $\slide$ below $\log x_2$, but this is not the case.
    The number of short chains (for the input $[1,x)$) is
    \begin{eqnarray*}
    \sum_{m=1}^{x^{2/5}} \pi(x/m)&\sim&
    \sum_{m=1}^{x^{2/5}} \frac{x}{m\log(x/m)} \\
    &=&
    \Theta\left(\sum_{m=1}^{x^{2/5}} \frac{x}{m\log x}\right)
    = \Theta\left(\frac{x}{\log x} \cdot \log(x^{2/5})\right)=\Theta(x).
    \end{eqnarray*}
    Given the running time, either data structure, binary heap or
    hash table, is reasonable for this approach.

Next, we consider generating primes as needed using a prime test.
In particular, given a starting point $n$, we find the smallest prime
larger than $n$ as follows:
\begin{itemize}
    \item If $n<B$ we use the precomputed list of primes to find the next
      prime in constant time.
      This implies storing prime indices instead of prime values.
      We distinquish indices from prime values by storing indices as negative
      numbers (so the prime $p_3=5$ is stored as $-3$).
    \item For larger $n$ we follow the ideas in \cite{Sorenson06}:
    \begin{enumerate}
        \item Sieve an interval of size $2\log n$ using primes up to $\log n$.
          This takes $O(\log n)$ time leaving $O(\log n/\log\log n)$ possible
          primes in the interval $(n,n+2\log n]$.
        \item Scan this interval for a possible prime, and then perform a
          base-$2$ strong pseudoprime test.
        \item If the test is passed, the number is likely prime, so do a prime test.
          In theory, we use a pseudosquares prime test which takes
          $O((\log n)^2)$ time in theory.
          In practice, we could use a Baille-PSW prime test, which works on
          64-bit integers, or a Miller-Rabin style test on a fixed, short list of
          appropriate bases.
    \end{enumerate}
    This is repeated as needed until a prime is found.
\end{itemize}
The average time to find a prime is dominated by the cost of the final prime test,
at $O((\log n)^2)$ time.
In practice, we used a fixed wheel of size 30, performed some fixed trial
divisions, and then used the Miller-Rabin test as described above,
taking something closer to $O((\log n)^2/\log\log n)$ heuristic time.

We also need a previous prime function, but the cost is the same.
We leave the details to the reader.
\begin{theorem}\label{thm:pqtest}
After setup,
    the priority queue algorithm takes $O(\Delta(\log x_2)^2)$ time
    to find and count all representations of integers in the interval
    $[x_1,x_2)$ where $\Delta=x_2-x_1$.
    The space used is $O(\sqrt{x_2\log x_2})$.
\end{theorem}
\begin{proof}
    After setup, the number of $\slide$ operations is $O(\Delta)$.
    From the discussion above, using the pseudosquares prime test we can
    perform each $\slide$ in $O((\log x_2)^2)$ time.
    The only significant use of space is to store the $O(M)$ chain objects,
    each of which uses $O(\log x_2)$ bits.
    This completes the proof.
\end{proof}

    \section{Computational Results\label{sec:comp}}
Several different practical implementations of the algorithms presented in \S\ref{sec:puzzlealg} and \S\ref{sec:pqalg} were created and run over several upper bounds, and we present both their results and timing information below. Values of the histogram function $h(k,x)$, where $x$ ranges over powers of ten\footnote{The authors are aware that Tables \ref{TAB:FREQDATA} and \ref{TAB:FISH} uses $x=10^{14}+145300$ in place of $x=10^{14}$. This will be corrected in a future revision of this paper.}, are given in Table \ref{TAB:FREQDATA}. We also give the average density of $f$ on these intervals, in Figure \ref{FIG:AVGREPS}.

\begin{table}
    \centering
    \begin{tabular}{c|rrrrrrrr}
    Count & $10^{2}$ & $10^{3}$ & $10^{4}$ & $10^{5}$ & $10^{6}$ & $10^{7}$ & $10^{8}$ & $10^{9}$ \\
    \hline
    0 & 46 & 520 & 5191 & 51462 & 518001 & 5205110 & 52209546 & 522955756 \\
    1 & 38 & 310 & 3290 & 34538 & 344100 & 3427038 & 34146573 & 340693986 \\
    2 & 14 & 140 & 1213 & 11236 & 111132 & 1099545 & 10950371 & 109272550 \\
    3 & 2 & 28 & 275 & 2396 & 22916 & 228659 & 2292360 & 23003362 \\
    4 & 0 & 0 & 29 & 323 & 3409 & 35009 & 353614 & 3584873 \\
    5 & 0 & 2 & 2 & 44 & 403 & 4197 & 42946 & 440748 \\
    6 & 0 & 0 & 0 & 1 & 37 & 412 & 4205 & 44623 \\
    7 & 0 & 0 & 0 & 0 & 2 & 28 & 356 & 3793 \\
    8 & 0 & 0 & 0 & 0 & 0 & 2 & 28 & 299 \\
    9 & 0 & 0 & 0 & 0 & 0 & 0 & 1 & 10 \\

    \hline
    \end{tabular}\vspace*{2ex}
    \begin{tabular}{c|rrrrr}
    Count & $10^{10}$ & $10^{11}$ & $10^{12}$ & $10^{13}$ & $10^{14}$ \\
    \hline
    0 & 5232998716 & 52334124696 & 523187120742 & 5229093083950 & 52255406559014 \\
    1 & 3402573440 & 34000179027 & 339871838389 & 3398259361550 & 33983734548218 \\
    2 & 1092269461 & 10929758863 & 109433799948 & 1096096081344 & 10980796355362 \\
    3 & 230955228 & 2319354726 & 23298396623 & 234059213343 & 2351331657325 \\
    4 & 36207780 & 365678388 & 3691321145 & 37239962166 & 375496312243 \\
    5 & 4493266 & 45727635 & 464500544 & 4710770197 & 47717060499 \\
    6 & 458971 & 4726851 & 48377037 & 493740792 & 5027735200 \\
    7 & 39944 & 415858 & 4291580 & 44107969 & 451927961 \\
    8 & 2980 & 31743 & 330053 & 3428047 & 35376934 \\
    9 & 204 & 2108 & 22539 & 235471 & 2452073 \\
    10 & 10 & 101 & 1332 & 14382 & 151480 \\
    11 & 0 & 4 & 66 & 757 & 8546 \\
    12 & 0 & 0 & 2 & 31 & 430 \\
    13 & 0 & 0 & 0 & 1 & 14 \\
    14 & 0 & 0 & 0 & 0 & 1 \\
    \hline
    \end{tabular}
    \caption{Frequency data up to $10^{14}$}
    \label{TAB:FREQDATA}
\end{table}

\begin{figure}
    \centering
    \includegraphics[height=3in]{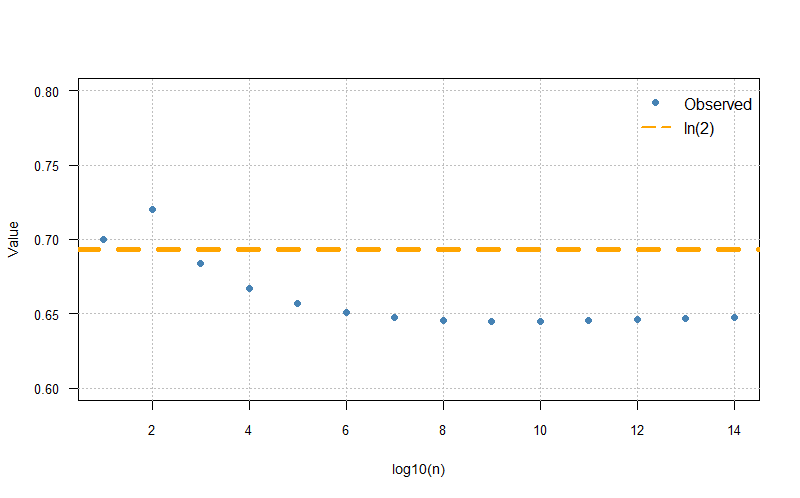}
    \caption{Average number of representations below $n$}
    \label{FIG:AVGREPS}
\end{figure} 

The puzzle algorithm from \S\ref{sec:puzzlealg} was implemented using Miller-Rabin deterministic primality testing, and ran in parallel on 560 AMD EPYC9734 cores at 2.20GHz. The interval $[10^{13},10^{14}]$ took 3 weeks, and the program needed 4 cumulative days to compute the interval $[1,10^{13}]$. We estimate that it would take on the order of one year to run this implementation on the interval $[10^{14},10^{15}]$, where we can heuristically expect to find the first integer with 15 gleeful representations.

The priority queue algorithm from \S\ref{sec:pqalg} was implemented using the C++ standard library \texttt{unordered\_map} as a hash table, and generated primes using an incremental sieve based on the sieve of Eratosthenes. This program ran on a single thread of an Intel Xeon E5-2650 v4 at 2.20GHz, running with an upper bound of $x=10^{12}$ and taking just over 26 days to terminate. Lower-order results include taking 3 days for $x=10^{11}$ and 6 hours cumulative for all smaller powers of ten.

We also created a version of the priority queue algorithm using a custom min-heap implementation with primality testing,  storing chain objects indexed by length in an array. The authors discussed splitting the min-heap into two separate structures, one to store short chains and the other to store long chains. This was of interest, since one could cache the short-chain array for a possible run-time improvement. We can estimate the total work done in a heap by integrating the number of gleeful representations of a specific length over the lengths included in the heap, so the short heap would require
\begin{equation}
    \label{EQN:BOUNDS}
    \sum_{m=1}^{k}\pi\left(\frac{x}{m}\right)\sim\int_{1}^{k}\frac{x}{m\log(x/m)}\mathrm{d}m=x\log\big\vert\log x-\log m\big\vert_{m=k}^{1}
\end{equation}
amount of work with cutoff $k$. In Moser's original paper, he used $k=\sqrt{x/\log x}$ as the bound on the maximum possible length of a gleeful number with sum $\leq x$ to get $x\log 2$ total representations. For the purposes of our work, if $k=(x/\log x)^c$ for some $c\in[0,1/2]$, then (\ref{EQN:BOUNDS}) simplifies to $-x\log(1-c+\frac{c\log\log x}{\log x})$. We want to have reasonable hope of a cache hit for our short chains, where the proportion of work done in cache is $-\log(1-c+\frac{c\log\log x}{\log x})/\log 2$. For illustration, if we wanted to get a cache hit on half of every update cycle when running with upper bound $x=10^{15}$, we would need to hold $x^{0.326}$ (where $c=(1-2^{-0.5})/(1-\frac{\log\log 10^{15}}{\log 10^{15}})$) chains, about 78000, in cache, well beyond the scope of our machines. For this reason, the separate heap structures were never implemented.

\subsection{Moser's Questions}
\label{sec:moser}
Recall the questions Moser posed in his 1963 paper:
\begin{quote}
    Since the average value of $f(n)$ is $\log 2$ it follows that $f(n)=0$ infinitely often. The following problems, among others, suggest themselves:
    \begin{enumerate}
        \item Is $f(n)=1$ infinitely often?
        \item Is $f(n)=k$ solvable for every $k$?
        \item Do the numbers for which $f(n)=k$ have a density for every $k$?
        \item Is $\limsup_{n\to\infty} f(n)=\infty$?
    \end{enumerate}
\end{quote}

In the spirit of Cram\'er's model on the distribution of primes, the frequencies in Table \ref{TAB:FREQDATA} give rise to a conjecture on a probabilistic model for Moser's function $f$. Namely, although not a true random probability distribution, $f(n)$ is roughly Poisson with mean $\lambda = \log 2$. Compare the histogram data for $x=10^{14}$ (restated for convienence) with the expected counts of a Poisson variable (Table \ref{TAB:FISH}). In this spirit, we can heuristically answer in the affirmative to Moser's four questions. 

\begin{table}
    \centering
    \begin{tabular}{c|rr}
    $k$ & $h(k)$ & $n\mathrm{Pr}[X=k]$ \\ \hline
    0 & 52255406573294 & 50000000080183 \\
    1 & 33983734548972 & 34657359083576 \\
    2 & 10980796355393 & 12011325367217 \\
    3 & 2351331657326 & 2775205437692 \\
    4 & 375496312243 & 480906456153 \\
    5 & 47717060499 & 66667790839 \\
    6 & 5027735200 & 7701765209 \\
    7 & 451927961 & 762636691 \\
    8 & 35376934 & 66077434 \\
    9 & 2452073 & 5089043 \\
    10 & 151480 & 352746 \\
    11 & 8546 & 22228 \\
    12 & 430 & 1284 \\
    13 & 14 & 68 \\
    14 & 1 & 3 \\ \hline
    \end{tabular}
    \caption{$h(k)$ is heuristically Poisson with mean $\lambda=\log 2$}
    \label{TAB:FISH}
\end{table}

In particular interest to Moser's second and fourth questions, OEIS A054859 lists the smallest integer having exactly $k$ representations. Prior to our work, this sequence was known to $k=13$; we present the first known integer with 14 representations as a sum of consecutive primes. Table \ref{TAB:OEIS} gives the complete updated OEIS table, and Table \ref{TAB:14reps} gives the 14 representations for our submitted integer. It is particularly interesting to note that 11 of the 14 records in Table \ref{TAB:OEIS} are themselves prime.

\begin{table}
    \centering
    \begin{tabular}{c|r}
        $k$ &                    $n$ \\ \hline
          0 &                      1 \\
          1 &                      2 \\
          2 &                      5 \\
          3 &                     41 \\
          4 &                 1\,151 \\
          5 &                    311 \\ 
          6 &	               34\,421 \\
          7 &        	      218\,918 \\
          8 &            3\,634\,531 \\
          9 &           48\,205\,429 \\
         10 &       1\,798\,467\,197 \\
         11 &      12\,941\,709\,050 \\ 
         12 &     166\,400\,805\,323 \\
         13 &  6\,123\,584\,726\,269 \\
         14 & 84\,941\,668\,414\,584 \\ \hline            
    \end{tabular}
    \caption{Smallest $n$ with $f(n)=k$ representations}
    \label{TAB:OEIS}
\end{table}

\begin{table}
    \centering
    \begin{tabular}{r|rr}
             $\ell$ &             $p_{\min}$ &             $p_{\max}$ \\ \hline
        2\,117\,074 &           21\,797\,833 &           58\,785\,359 \\
           361\,092 &          231\,753\,581 &          238\,710\,779 \\
           288\,268 &          291\,853\,531 &          297\,473\,801 \\
           199\,390 &          424\,030\,259 &          427\,989\,799 \\
           112\,544 &          753\,590\,641 &          755\,886\,067 \\
            73\,026 &       1\,162\,407\,049 &       1\,163\,930\,791 \\
            68\,854 &       1\,232\,927\,929 &       1\,234\,369\,457 \\
                296 &     286\,965\,092\,209 &     286\,965\,099\,727 \\
                294 &     288\,917\,235\,553 &     288\,917\,243\,497 \\
                206 &     412\,338\,193\,609 &     412\,338\,198\,731 \\
                146 &     581\,792\,247\,697 &     581\,792\,251\,207 \\
                 86 &     987\,693\,817\,667 &     987\,693\,819\,859 \\
                 26 &  3\,266\,987\,246\,389 &  3\,266\,987\,247\,019 \\
                  2 & 42\,470\,834\,207\,273 & 42\,470\,834\,207\,311 \\ \hline
    \end{tabular}
    \caption{The 14 representations for $g=84\,941\,668\,414\,584$.}
    \label{TAB:14reps}
\end{table}

With regard to Moser's first and third questions, we can conjecture that the density of numbers $n$ for which $f(n)=k$ follows the Poisson distribution's density: asymptotically $\frac{(\log 2)^k}{2\cdot k!}$. Indeed, we expect half of the integers will have no representation as a sum of consecutive primes, $\log 2/2$ (roughly 1/3) of the integers will have exactly one representation, and so forth.


    \section{Acknowledgements\label{SEC:ACKNOWLEDGE}}
	This work was supported in part by a grant from the Holcomb Awards Committee. 
    We thank Frank Levinson for his generous support
    of Butler's research computing infrastructure.

    Thanks to Ankur Gupta and Jonathan Webster for several
    helpful discussions.

    The second author was an undergraduate student when this work
    was done.

 \bibliographystyle{plain}

\end{document}